\pdfoutput=1\relax
\documentclass[english]{article}
\usepackage{meta/preamble}
\title{Local structure of smooth \texorpdfstring{$p$}{p}-adic analytic Artin stacks}
\date{\today}
\author{Amos Kaminski\thanks{Department of Mathematics, Weizmann Institute of Science \\ amos.kaminski@weizmann.ac.il}}
\begin{document}

\setlength{\abovedisplayskip}{3pt}
\setlength{\belowdisplayskip}{3pt}
\setlength{\abovedisplayshortskip}{0pt}
\setlength{\belowdisplayshortskip}{0pt}

\maketitle
\begin{center}
\end{center}
\begin{abstract}
We prove \cite[Conjecture~5.17]{Clausen} on the local light--profinite structure of smooth $p$-adic analytic Artin stacks.
The argument proceeds in several reductions.
First, by proving a generalization of van~Dantzig theorem for groupoids, we reduce the conjecture to the compact Hausdorff case.
This reduces the conjecture to the statement that the geometric realization of a groupoid object whose object and morphism spaces are light profinite and whose source and target maps are open is light profinite.
Next, we simplify the groupoid by constructing a closed skeleton; after quotienting by a clopen subgroupoid, the remaining problem reduces to proving that a profinite family of finite groups can be presented as an inverse limit of finite families of finite groups.
As observed by Clausen immediately after \cite[Conjecture~5.17]{Clausen}, our result implies in particular that smooth $p$-adic analytic Artin stacks are \emph{$!$-good}.
\end{abstract}

\setcounter{tocdepth}{2}
$\\$
\begin{figure}[H]
  \centering{}
  \setlength{\fboxsep}{-5pt}
  \setlength{\fboxsep}{5pt}
  \frame{\includegraphics[scale=0.14]{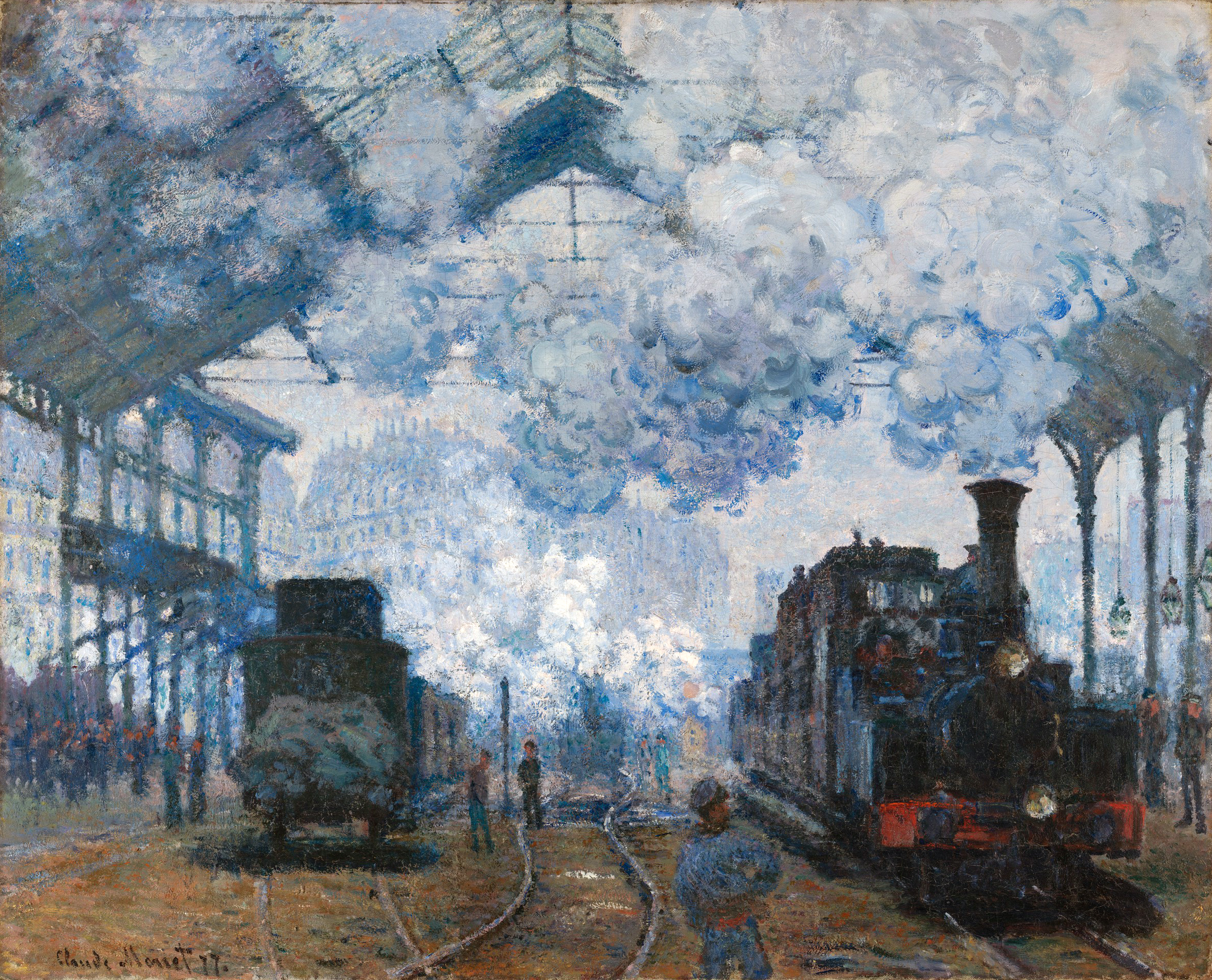}}
  \caption{\footnotesize 
Claude Monet - The Gare Saint-Lazare, Arrival of a Train\\}
    
\end{figure}
\newpage
\begingroup
\renewcommand{\baselinestretch}{1.2}\selectfont
\setlength{\parskip}{6pt}
\tableofcontents
\endgroup
\newpage
\section{Introduction}

In \cite{Clausen}, Clausen conjectured that smooth $p$-adic analytic Artin stacks become
\emph{\'etale-locally light profinite} after passage to light condensed anima.
Concretely, let
\[ L:\Sh(\Man_{\mathbb{Q}_p})\longrightarrow \mathrm{CondAn}^{\mathrm{light}} \]
be the left Kan extension of the functor sending a $p$-adic manifold to its underlying topological space,
viewed on the light site (i.e.\ through maps from light profinite sets).
Then \cite[Conjecture~5.17]{Clausen} predicts that for any smooth $p$-adic analytic Artin stack $X$,
\'etale-locally the associated light condensed anima $L(X)$ is a light profinite anima.

\begin{theoremA}[{\cite[Conjecture~5.17]{Clausen}} -- Theorem \ref{thm:clausen-conj} in the text]\label{thm:conj517_intro}
Let $X$ be a smooth $p$-adic analytic Artin stack. Then \'etale-locally the associated light condensed anima $L(X)$ is a light profinite anima.
\end{theoremA}

\smallskip
In this paper we prove Conjecture~5.17.
We work throughout in the setting of pyknotic spaces $\PykS$ \cite{BH19}, which is convenient for several reasons.
First, light condensed anima embed into $\PykS$ by restriction along the light site.
Second, the global sections functor on $\PykS$ admits both a left and a right adjoint; this will be used repeatedly in descent and detection arguments.
For instance, we will use the following conservativity statement.

\begin{theoremB}[Theorem \ref{thm:Gamma-conservative-1type} in the text]
Let $f:X\to Y$ be a morphism between \emph{coherent} $1$-truncated objects in $\PykS$.
Then $f$ is an equivalence if and only if $\Gamma(f):\Gamma(X)\to\Gamma(Y)$ is an equivalence of spaces.
\end{theoremB}

Finally, since $\Pro(\mathcal{S}_\pi)$ embeds fully faithfully into $\PykS$ (see \cite[Example 3.3.10]{BH19}), we can form profinite completions inside $\PykS$.
To conclude lightness we use the following characterization.

\begin{propositionC}[Proposition \ref{light} in the text]
Let $X\in \PykS$ be a profinite anima, i.e.\ $X\in \Pro(\mathcal{S}_\pi)$.
Assume that $X$ lies in the essential image of the inclusion
\[
\CondAn^{\mathrm{light}}\hookrightarrow \PykS .
\]
Then $X$ is a \emph{light profinite anima}.
\end{propositionC}

\medskip
The proof has two main parts.

\smallskip\noindent
\textbf{Step 1: Reduction to a compact groupoid presentation.}
We first pass to an \'etale cover of $X$ on which the stack admits a groupoid presentation whose object space is compact Hausdorff.
We then prove a generalization of Van Dantzig theorem for groupoids.

\begin{theoremD}[Theorem \ref{VanDan} in the text]
Let $G \rightrightarrows G_0$ be a groupoid object in $\Pyk(\mathcal S)$ such that the source and target maps
\[
s,t: G_1\to G_0
\]
are \emph{open maps}. Assume $G_1$ is locally profinite and $G_0$ is profinite.
Then for every \emph{open} neighborhood $U$ of $G_0$ in $G_1$, there exists a \emph{Hausdorff compact open} subgroupoid
$H \subseteq G$ such that
\[
G_0 \subseteq H_1 \subseteq U.
\]
(In particular, $H$ is wide in the sense that it contains all units.)
\end{theoremD}

This allows us to replace the resulting groupoid by a compact, Hausdorff, wide subgroupoid.
After this reduction, both the object and morphism spaces of the groupoid are light profinite.

\smallskip\noindent
\textbf{Step 2: Realization and Profiniteness.}
For the compact groupoid obtained in Step~1, we show that its geometric realization is profinite.
We begin by reducing to a \emph{closed skeleton}.

\begin{lemmaE}[Lemma \ref{lem:skeletal-replacement} in the text]
Let $\mathcal{G}$ be a groupoid object whose object and morphism spaces are light profinite, and whose source and target maps are open.
Then there exists a groupoid $\mathcal{X}$ and a functor $p:\mathcal{X}\to \mathcal{G}$ such that:
\begin{enumerate}[label=(\roman*),leftmargin=2.2em]
    \item $\Obj(\mathcal{X})$ and $\Mor(\mathcal{X})$ are light profinite;
    \item the induced map on geometric realizations
    \[
    |p|:\,|\mathcal{X}|\longrightarrow |\mathcal{G}|
    \]
    is an equivalence;
    \item $\mathcal{X}$ is skeletal in the sense that for any two distinct objects
    $x\neq y$ in $\Obj(\mathcal{X})$ one has $\Hom_{\mathcal{X}}(x,y)=\varnothing$.
\end{enumerate}
\end{lemmaE}

Using Theorem~\ref{VanDan}, we construct a basis of clopen subgroupoids on the skeleton, and refine it further to a basis of normal clopen subgroupoids.
Passing to the corresponding quotients reduces the remaining key point to the following observation:
A profinite family of finite groups can be expressed as an inverse limit of finite families of finite groups.

\medskip
As Clausen notes immediately after stating the conjecture, Theorem~\ref{thm:conj517_intro} implies that smooth $p$-adic analytic Artin stacks are \emph{$!$-good}.

\smallskip
\noindent\textbf{Organization.}
Section~2 recalls the definitions of light condensed anima, \'etale morphisms, and the statement of Conjecture~5.17.
Section~3 proves the technical results used later, including the profiniteness criterion for groupoids.
Section~4 combines these inputs to prove the conjecture.

\subsection*{Acknowledgments}
First and foremost, I would like to express my heartfelt gratitude to Akhil Mathew for many insightful discussions and for carefully going over earlier versions of my work, pointing out numerous errors along the way. I am equally grateful to my PhD advisor, Shachar Carmeli, for his guidance and support. I owe special thanks to Yuval Lotenberg for many helpful discussions and for providing detailed feedback on my draft. I wish to extend my appreciation to Tomer Schlank, Shauli Ragimov, and Noam Zimhoni for useful conversations.
Finally, I would like to thank Dustin Clausen for formulating the conjecture and bringing it to my attention, and for pointing out an important issue in an earlier version.

\section{Background}
In this section, we recall the minimal background from \cite{Clausen} needed to state Clausen’s conjecture in a self-contained way. We do not claim originality for any of the statements in this section.

\subsection{Light condensed anima}

We begin by recalling the notion of a light profinite set and a light condensed anima

\begin{definition}[{\cite[Definition 4.2.]{Clausen}}]
A \emph{light profinite set} is a topological space which can be written as a sequential inverse limit
of finite discrete spaces.
\end{definition}

Let $\Pro(\mathrm{fin})^{\mathrm{light}}$ denote the category of light profinite sets with continuous maps.
We can now define light condensed anima as sheaves on this site.

\begin{definition}[{\cite[Definition 4.6.]{Clausen}}]
A \emph{light condensed anima} is a hypercomplete sheaf for the effective topology on $\Pro(\mathrm{fin})^{\mathrm{light}}$, i.e.\ a functor
\[
X \colon (\Pro(\mathrm{fin})^{\mathrm{light}})^{\mathrm{op}} \longrightarrow \mathrm{An}
\]
satisfying the following two conditions:
\begin{enumerate}
\item For every finite family $\{S_i\}_{i\in I}$ of light profinite sets, the natural map
\[
X\Bigl(\bigsqcup_{i\in I} S_i\Bigr) \longrightarrow \prod_{i\in I} X(S_i)
\]
is an equivalence.
\item For every hypercover $S_\bullet \to S$ in $\Pro(\mathrm{fin})^{\mathrm{light}}$, the natural map
\[
X(S) \longrightarrow \varprojlim_{[n]\in\Delta} X(S_n)
\]
is an equivalence.
\end{enumerate}
The full subcategory of such functors is denoted $\mathrm{CondAn}^{\mathrm{light}}$.
\end{definition}
We next record the canonical functors relating light condensed anima to pyknotic spaces.
\begin{remark}
There is a natural fully faithful inclusion
\[
i:\Pro(\mathrm{fin})^{\mathrm{light}} \hookrightarrow \Pro(\mathrm{fin}).
\]
Since pyknotic spaces can be described as hypercomplete sheaves on $\Pro(\mathrm{fin})$, $i$ induces a restriction functor
\[
i^*:\Pyk(\mathcal{S}) \to \mathrm{CondAn}^{\mathrm{light}},
\]
which admits a fully faithul left adjoint, given by left Kan extension along $i$, and denoted
\[
i_*:\mathrm{CondAn}^{\mathrm{light}} \to \Pyk(\mathcal{S}).
\]

\end{remark}

\subsection{\texorpdfstring{$p$}{p}-adic analytic smooth Artin stacks}
We now introduce the notion of a light profinite anima, which will be central in formulating Clausen’s conjecture.
\begin{definition}[{\cite[Example 4.24.]{Clausen}}]
An object $Y \in \mathrm{CondAn}^{\mathrm{light}}$ is called a \emph{light profinite anima}
if it can be written as a sequential inverse limit of $\pi$-finite anima.
\end{definition}

Equivalently, light profinite anima form the essential image of the category of countable
pro-objects in $\pi$-finite anima under the natural fully faithful embedding into
$\mathrm{CondAn}^{\mathrm{light}}$ (\cite[Lemma 4.22]{Clausen}).
We now construct, from a sheaf of anima on $\Man_F$, an associated light condensed anima.
Concretely, it is obtained as the left Kan extension of the functor sending a manifold to its underlying topological space.

\begin{proposition}[{\cite[Proposition 5.3]{Clausen}}]
Let $F$ be a local field.
\begin{enumerate}
\item There is a unique colimit-preserving functor
\[ 
L:\Sh(\Man_F)\to \mathrm{CondAn}^{\mathrm{light}}
\]
whose restriction to $\Man_F$ is the functor
\[
M \mapsto M^{\mathrm{top}} : \Man_F \to \mathrm{CondAn}^{\mathrm{light}},
\]
where $M^{\mathrm{top}}$ denotes the underlying topological space of $M$.
By abuse of notation, denote this functor also by
\[
X \mapsto \underline{X}
\]
(or sometimes, if we're being really abusive, just $X\mapsto X$ with the context
being implicit).
\item This functor $X\mapsto \underline{X}$ preserves pullbacks of diagrams of the form
$X \to S \leftarrow Y$ whenever $X\to S$ can, locally on $S$, be written as a colimit
of submersions of manifolds $X_i \to S$.
\end{enumerate}
\end{proposition}
We will now define our main object of interest, namely \emph{$F$-analytic smooth Artin stacks}, i.e.\ stacks admitting a manifold chart.
Recall that a morphism \(Y\to X\) in \(\Sh(\Man_F)\) is a \emph{representable submersion} if for every map \(N\to X\) from an \(F\)-analytic manifold \(N\), the pullback \(Y\times_X N \to N\) is a submersion of \(F\)-analytic manifolds. 
Surjectivity is understood as being an \emph{effective epimorphism} in the \(\infty\)-topos \(\Sh(\Man_F)\).

\begin{definition}[{\cite[Definition~5.11]{Clausen}}]
Let \(F\) be a local field.
An \emph{\(F\)-analytic smooth Artin stack} is a sheaf
\(X \in \Sh(\Man_F)\) admitting a surjective representable submersion
\[
p\colon M \longrightarrow X
\]
from an \(F\)-analytic manifold \(M\).
\end{definition}

\subsection{\'Etale locality}

The category $\mathrm{CondAn}^{\mathrm{light}}$ carries an intrinsic notion of \'etale morphisms.
Given an object $S \in \mathrm{CondAn}^{\mathrm{light}}$, one may speak of properties holding
\emph{\'etale-locally on $S$}, meaning that there exists an \'etale cover $\{U_i \to S\}$
such that the property holds after base change to each $U_i$.
We will now give two equivalent characterizations of \'etale morphisms.

\begin{theorem}[{\cite[Theorem~4.15]{Clausen}}]
\label{etale}
Let $f\colon X \to Y$ be a map in $\mathrm{CondAn}^{\mathrm{light}}$. The following properties are equivalent.
\begin{enumerate}
\item  For all towers $(S_n)_n$ of light profinite sets with limit $S=\varprojlim_n S_n$, the map $f$ has the unique right lifting property with respect to the map (of pro-objects)
\[
S \to \text{``}\!\varprojlim_n\! \text{'' } S_n.
\]
In other words,
\[
\colim_n X(S_n)\;\xrightarrow{\sim}\; X(S)\times_{Y(S)}\Bigl(\colim_n Y(S_n)\Bigr).
\]
\item The functor $(\Pro(\mathrm{fin})^{\mathrm{light}}_{/Y})^{\mathrm{op}}\to \mathrm{An}$ represented by $f\colon X\to Y$ preserves countable filtered colimits.

\end{enumerate}
\end{theorem}

\begin{definition}[{\cite[Definition~4.16]{Clausen}}]
A map $X\to Y$ in $\mathrm{CondAn}^{\mathrm{light}}$ is called \emph{\'etale} if it satisfies the equivalent properties listed in the above theorem.
\end{definition}
Here are some useful properties of \'etale maps:
\begin{lemma}[{\cite[Lemma~4.17]{Clausen}}] \label{etaleprop}
\begin{enumerate}
\item The collection of \'etale maps of light condensed anima is closed under pullbacks and compositions, and contains all isomorphisms. It is also closed under passage to diagonals; equivalently, if $X\to Y\to Z$ are maps of light condensed anima such that $X\to Z$ and $Y\to Z$ are \'etale, then $X\to Y$ is \'etale.
\item For $S\in \mathrm{CondAn}^{\mathrm{light}}$, the full subcategory of $\mathrm{CondAn}^{\mathrm{light}}_{/S}$ spanned by the \'etale maps is closed under finite limits and arbitrary colimits.
\item A map $X\to Y$ of light condensed anima is \'etale if and only if its $n$th relative Postnikov truncation is \'etale for all $n$.
\item If $\{S_i\to S\}_{i\in I}$ is a collection of maps of light condensed anima such that $\bigsqcup_{i\in I} S_i\to S$ is surjective, then a map $X\to S$ is \'etale if and only if the pullback $X\times_S S_i\to S_i$ is \'etale for all $i\in I$.
\end{enumerate}
\end{lemma}
We will now state the conjecture:

\begin{conjecture}[{\cite[Conjecture 5.17]{Clausen}}]
Let $X$ be a $p$-adic analytic smooth Artin stack.
Then \'etale-locally the associated light condensed anima $L(X)$
is a light profinite anima.
\end{conjecture}

\section{Preparatory Theorems}
\subsection{Characterization of light profinite anima}
We now record a useful characterization of light profinite objects inside $\Pro(\mathcal S_\pi)$ in terms of $\omega_1$-cocompactness.
\begin{lemma}
Under the inclusion
\(\Pro^{\mathrm{light}}(\mathcal S_\pi)\hookrightarrow \Pro(\mathcal S_\pi)\),
an object \(X\in \Pro(\mathcal S_\pi)\) is \(\omega_1\)-cocompact
if and only if \(X\) is light profinite, i.e. a pro-object
presented by a countable cofiltered limit of objects of \(\mathcal S_\pi\).
Equivalently,
\[
\Pro^{\mathrm{light}}(\mathcal S_\pi)=\Pro(\mathcal S_\pi)^{\omega_1\text{-cocpt}}
\]
as full subcategories of \(\Pro(\mathcal S_\pi)\).
\end{lemma}

\begin{proof}
Recall the standard equivalence
\[
\Pro(\mathcal S_\pi)\;\simeq\; \Ind(\mathcal S_\pi^{\mathrm{op}})^{\mathrm{op}}.
\]
Under this equivalence, \(X\in \Pro(\mathcal S_\pi)\) is \(\omega_1\)-cocompact
if and only if its opposite \(X^{\mathrm{op}}\in \Ind(\mathcal S_\pi^{\mathrm{op}})\) is
\(\omega_1\)-compact.

\smallskip

\noindent\emph{(\(\Rightarrow\))} Suppose \(X\) is \(\omega_1\)-cocompact. Then
\(X^{\mathrm{op}}\) is \(\omega_1\)-compact in \(\Ind(\mathcal S_\pi^{\mathrm{op}})\).
By Lurie’s characterization of \(\kappa\)-compact objects in \(\Ind\)-categories
\cite[Proposition~5.3.4.17]{HTT}, \(X^{\mathrm{op}}\) is a retract of a \(\omega_1\)-small
colimit of representables. Since \(\omega_1\)-small means \emph{countable},
we can write
\[
X^{\mathrm{op}}\;\text{ is a retract of }\;\colim y(c_i)
\]
for some countable category \(I\) and objects \(c_i\in \mathcal S_\pi^{\mathrm{op}}\); this exhibits \(X\)
as a retract of the corresponding countable limit
\[
\lim_{i\in I^{\mathrm{op}}} c_i
\]
in \(\Pro(\mathcal S_\pi)\). Hence since light profinite objects are closed under retracts and countable limits, \(X\) is also light profinite.

\smallskip

\noindent\emph{(\(\Leftarrow\))} Conversely, suppose \(X\) is light profinite,
so \(X\) is a pro-object of the form \(\varprojlim_{i\in J} c_i\)
with \(J\) countable cofiltered and \(c_i\in\mathcal S_\pi\).
Passing to opposites, \(X^{\mathrm{op}}\) is a filtered colimit
\[
\varinjlim_{i\in J^{\mathrm{op}}} y(c_i)
\]
in \(\Ind(\mathcal S_\pi^{\mathrm{op}})\), where \(J^{\mathrm{op}}\) is countable filtered.
Again by \cite[Proposition~5.3.4.17]{HTT}, any \(\omega_1\)-small colimit of
representables is \(\omega_1\)-compact. Thus \(X^{\mathrm{op}}\) is \(\omega_1\)-compact, hence
\(X\) is \(\omega_1\)-cocompact in \(\Pro(\mathcal S_\pi)\).

\smallskip

Therefore \(X\) is \(\omega_1\)-cocompact if and only if \(X\) is light profinite.
\end{proof}
We now relate the two notions of lightness by showing that any profinite anima which is already a light condensed anima is automatically light profinite.

\begin{proposition}\label{light}
Let $X\in \PykS$ be a profinite anima, i.e.\ $X\in \Pro(\mathcal{S}_\pi)$.
Assume that $X$ lies in the essential image of the inclusion
\[
\CondAn^{\mathrm{light}}\hookrightarrow \PykS .
\]
Then $X$ is a \emph{light profinite anima}.
\end{proposition}

\begin{proof}
By the previous lemma, it suffices to show that $X$ is $\omega_1$-cocompact in $\Pro(\mathcal{S}_\pi)$.
Let $(P_i)_{i\in I}$ be an $\omega_1$-cofiltered diagram in $\Pro(\mathcal{S}_\pi)$ and write $P=\varprojlim_{i\in I} P_i$.

Since $X$ lies in the essential image of $\CondAn^{\mathrm{light}}\hookrightarrow \PykS$, the value $X(P)$ is computed by left Kan extension from the full subcategory of light objects. Concretely,
\[
X(P)\;\simeq\;\colim_{(K\to P)} X(K),
\]
where the colimit ranges over light objects $K$ equipped with a map to $P$.

$K$ is light hence from the previous lemma we see that $K$ is $\omega_1$-cocompact in $\Pro(\mathcal{S}_\pi)$ hence $\omega_1$-compact in $\Pro(\mathcal{S}_\pi)^{\text{op}}$, therefore 
\[
X(P)\;\simeq\;\colim_{(K\to P)} X(K)
\;\simeq\;\colim_{i\in I}\;\colim_{(K\to P_i)} X(K)
\;\simeq\;\colim_{i\in I} X(P_i).
\]
Thus $X$ carries $\omega_1$-cofiltered limits in $\Pro(\mathcal{S}_\pi)$ to colimits, i.e.\ $X$ is $\omega_1$-cocompact. Hence $X$ is a light profinite anima.
\end{proof}
\begin{remark}
    This is a straightforward generalization of the characterization of light condensed sets found in \cite{ConCom} just before Proposition 3.15
\end{remark}


\subsection{Global sections detect equivalences of coherent 1-pyknotic anima}
We now recall the internal notions of full faithfulness and essential surjectivity for morphisms of groupoid objects.
\begin{definition}
    Let $f:X_\bullet\to  Y_\bullet$ be a morphism between groupoid objects in a topos:
    \begin{itemize}
        \item $f$ is called internally fully faithful if the diagram
\[
\begin{tikzcd}
X_1 \ar[r] \ar[d] & Y_1 \ar[d] \\
X_0 \times X_0 \ar[r] & Y_0 \times Y_0
\end{tikzcd}
\] is a pullback diagram
\item $f$ is internally essentially surjective if the map
\[
k: X_0 \times_{Y_0} Y_1 \longrightarrow Y_0
\]
is an effective epimorphism.
    \end{itemize}
\end{definition}

The following theorem shows that internal full faithfulness and essential surjectivity are exactly the conditions needed in $\PykS$ for $|f|$ to be an equivalence.
\begin{theorem}\label{functor_theorem}
Let $f:X_\bullet\to Y_\bullet$ be a morphism of groupoid objects in $\PykS$.
Then the geometric realization $|f|:|X_\bullet|\to |Y_\bullet|$ is an equivalence if and only if $f$ is essentially surjective and fully faithful.
\end{theorem}

\begin{proof}
Since $\Pyk(S)\simeq \Sh_{\mathrm{eff}}(\text{EStn})$, where $\text{EStn}$ are the Stonean topological spaces, equivalences can be checked after evaluation on each $T\in\text{EStn}$.
Thus it suffices to check that $|f|$ is an equivalence after evaluation at each extremely disconnected space $T$.

By \cite[Example 2.2.11]{BH19}, evaluation at $T$ commutes with geometric realization, i.e.
\[
\text{ev}_T(|X_\bullet|)\simeq |\text{ev}_T(X_\bullet)| \qquad\text{and}\qquad
\text{ev}_T(|Y_\bullet|)\simeq |\text{ev}_T(Y_\bullet)|.
\]
Moreover $\text{ev}_T$ commutes with finite limits.

As a consequence a morphism in $\PykS$ is an effective epimorphism if and only if its evaluation at every extremely disconnected space is surjective.
Since the conditions “essentially surjective” and “fully faithful” for a morphism of groupoid objects are expressed using finite limits and effective epimorphisms, it follows that $f$ is essentially surjective (resp.\ fully faithful) if and only if $\text{ev}_T(f)$ is so for every extremely disconnected $T$.

Therefore the claim reduces to the corresponding statement for the evaluated groupoid objects $\text{ev}_T(X_\bullet),\text{ev}_T(Y_\bullet)$, which live in anima. In anima, this is the usual fact that a functor of $\infty$-categories is an equivalence if and only if it is essentially surjective and fully faithful.
\end{proof}

To reduce later equivalence checks to a statement about spaces, we prove fact that $\Gamma$ is conservative on coherent $1$-truncated pyknotic anima.
\begin{theorem}\label{thm:Gamma-conservative-1type}
Let $f:X\to Y$ be a morphism between \emph{coherent} $1$-truncated objects in $\PykS$.
Then $f$ is an equivalence if and only if $\Gamma(f):\Gamma(X)\to\Gamma(Y)$ is an equivalence of spaces.
\end{theorem}
\begin{proof}
The forward direction is immediate. We prove the reverse. Assume $\Gamma(f)$ is an equivalence in $\mathcal{S}$.

\smallskip

\noindent\textbf{Step 1: Present $X$ and $Y$ by coherent covers.}
Since $X$ is coherent, there exists an effective epimorphism $p:U\twoheadrightarrow X$ where $U$ is coherent and $0$-truncated (a finite disjoint union of profinite sets).
Let $U_\bullet$ be the Čech nerve of $p$. Since $X$ is $1$-truncated, $U_\bullet$ is a groupoid object on $0$-truncated objects, and we have an equivalence $X \simeq |U_\bullet|$.
Similarly, choose an effective epimorphism $q:V\twoheadrightarrow Y$ with $V$ coherent $0$-truncated, giving $Y \simeq |V_\bullet|$.

\smallskip

\noindent\textbf{Step 2: Construct the refinement}
We cannot assume a direct map between the simplicial objects $U_\bullet$ and $V_\bullet$. Instead, we construct a common refinement. Define
\[
W := U \times_Y V.
\]
Since $q:V\to Y$ is an effective epimorphism, its base change $W \to U$ is an effective epimorphism. Consequently, the composite $W \to U \to X$ is an effective epimorphism.
Let $W_\bullet$ be the Čech nerve of $W \to X$. We have a diagram of simplicial objects
\[
U_\bullet \xleftarrow{\;\;\alpha_\bullet\;\;} W_\bullet \xrightarrow{\;\;\beta_\bullet\;\;} V_\bullet
\]
where $\alpha_\bullet$ and $\beta_\bullet$ are induced by the projections $W\to U$ and $W\to V$.
This induces a diagram on geometric realizations:
\[
X \simeq |U_\bullet| \xleftarrow{\;\sim\;} |W_\bullet| \xrightarrow{\;|\beta_\bullet|\;} |V_\bullet| \simeq Y.
\]
The map $|\alpha_\bullet|$ is an equivalence because $W$ and $U$ both cover $X$ (so $W_\bullet$ is a refinement of the atlas $U_\bullet$).
Under these identifications, the original map $f$ corresponds to $|\beta_\bullet| \circ |\alpha_\bullet|^{-1}$.
Thus, to prove $f$ is an equivalence, it suffices to prove $|\beta_\bullet|$ is an equivalence.

\smallskip

\noindent\textbf{Step 3: Reduce to properties of internal groupoids.}
Applying $\Gamma$ and using the fact that $\Gamma$ commutes with colimit(as a left adjoint to the indiscrete functor), we have equivalences $\Gamma(X)\simeq |\Gamma(W_\bullet)|$ and $\Gamma(Y)\simeq |\Gamma(V_\bullet)|$.
The hypothesis that $\Gamma(f)$ is an equivalence implies that
\[
\Gamma(\beta_\bullet): \Gamma(W_\bullet) \longrightarrow \Gamma(V_\bullet)
\]
induces an equivalence on realizations. Therefore, $\Gamma(\beta_\bullet)$ is an \textbf{equivalence of groupoids}.

\smallskip

\noindent\textbf{Step 4: Apply conservativity levelwise.}
We now show $\beta_\bullet$ is an equivalence of groupoids in hence $f$ is an isomorphism. Note that all objects $W_n, V_n$ are coherent and $0$-truncated.

\textit{(i) Internal Full Faithfulness.}
$\beta_\bullet$ is internally fully faithful if the diagram
\[
\begin{tikzcd}
W_1 \ar[r] \ar[d] & V_1 \ar[d] \\
W_0 \times W_0 \ar[r] & V_0 \times V_0
\end{tikzcd}
\]
is cartesian. Let $P := (W_0 \times W_0) \times_{(V_0 \times V_0)} V_1$. We must show the canonical map $h: W_1 \to P$ is an isomorphism.
$P$ is a finite limit of coherent $0$-truncated objects, hence coherent $0$-truncated.
Applying $\Gamma$, we see $\Gamma(h)$ is an isomorphism because $\Gamma(\beta_\bullet)$ is fully faithful (a bijection on hom-sets).
By \cite[Corollary 3.10.]{DerivedStoneEmbedding}, since $h$ is a map between coherent $0$-truncated objects and $\Gamma(h)$ is an isomorphism, $h$ is an isomorphism in $\PykS$.

\textit{(ii) Internal Essential Surjectivity.}
$\beta_\bullet$ is internally essentially surjective if the map
\[
k: W_0 \times_{V_0} V_1 \longrightarrow V_0
\]
is an effective epimorphism. The source and target are coherent $0$-truncated.
Applying $\Gamma$, $\Gamma(k)$ is surjective because $\Gamma(\beta_\bullet)$ is essentially surjective.
But it is a known fact that a map between compact Hausdorff spaces is an effective epimorphism if and only if it is surjective.

\smallskip

\noindent\textbf{Conclusion.}
Since $\beta_\bullet$ is internally fully faithful and essentially surjective, the induced map on realizations $|\beta_\bullet|: |W_\bullet| \to |V_\bullet|$ is an equivalence in $\mathcal{X}$.
It follows that $f \simeq |\beta_\bullet| \circ |\alpha_\bullet|^{-1}$ is an equivalence.
\end{proof}
Combining the conservativity of $\Gamma$ with the classical Whitehead criterion for $1$-types, we deduce the following corollary.
\begin{corollary}\label{whitehead}
    Let $f:X\to Y$ be a morphism between \emph{coherent} $1$-truncated objects in $\PykS$.
Then $f$ is an equivalence if and only if $\pi_0(f)$ is an isomorphism and the induced map
$ f_*: \pi_1(X, x_0) \to \pi_1(Y, f(x_0)) $
is an isomorphism for every basepoint $x_0 \in X$
\end{corollary}

\subsection{\texorpdfstring{$F$}{F}-analytic Lie Groupoid}

We fix the following notion of an $F$-Lie groupoid.
\begin{definition}\label{def:padic-lie-groupoid}
Let $G \rightrightarrows G_0$ be a groupoid object in $\Pyk(\mathcal S)$, with space of objects $G_0$ and space of arrows $G_1$.
We say that $G$ is a \emph{$F$- Lie groupoid} if:
\begin{enumerate}
    \item $G_0$ and $G_1$ are $F$-analytic  manifolds;
    \item the source and target maps $s,t \colon G_1 \to G_0$ are submersions;
    \item the structure maps
    \[
        u \colon G_0 \to G_1,\qquad
        i \colon G_1 \to G_1,\qquad
        m \colon G_1 \times_{G_0} G_1 \to G_1
    \]
    (unit, inversion, and multiplication) are smooth, where the fiber product admits a manifold structure since $s$ and $t$ are submersions.
\end{enumerate}
\end{definition}

We next explain how an $F$-analytic smooth Artin stack can be presented by an analytic Lie groupoid via its \v{C}ech nerve.
\begin{theorem}[Presentation by an analytic Lie groupoid]\label{thm:smooth-artin-stacks-are-1truncated}
Let \(F\) be a local field and let \(X\in \Sh(\Man_F)\) be an \(F\)-analytic smooth Artin stack. 
Thus there exists a surjective representable submersion
\[
p\colon M \twoheadrightarrow X
\]
with \(M\in \Man_F\). Then:
\begin{enumerate}
    \item The object \(X\) is \(1\)-truncated.
    \item The \v{C}ech nerve \(M^{\bullet} = M^{\times_X (\bullet+1)}\) is a simplicial \(F\)-analytic manifold, and \(X\simeq |M^{\bullet}|\).
    \item The simplicial manifold \(M^{\bullet}\) is equivalent to the nerve of a lie groupoid
    \[
    R:=M\times_X M \rightrightarrows M,
    \]
\end{enumerate}
\end{theorem}

\begin{proof}
Consider the induced map
\[
q := p\times p \colon M\times M \longrightarrow X\times X .
\]
Since effective epimorphisms are stable under products in an \(\infty\)-topos, \(q\) is an effective epimorphism.

Pulling back the diagonal \(\Delta_X\colon X\to X\times X\) along \(q\) yields a cartesian square
\[
\begin{tikzcd}
X\times_{X\times X}(M\times M) \ar[r] \ar[d] & M\times M \ar[d,"q"] \\
X \ar[r,"\Delta_X"] & X\times X .
\end{tikzcd}
\]
There is a canonical identification
\[
X\times_{X\times X}(M\times M)\;\simeq\; M\times_X M,
\]
under which the base-changed morphism is the natural map
\[
M\times_X M \longrightarrow M\times M .
\]
In particular, the base change of \(\Delta_X\) along \(q\) is a morphism between \(0\)-truncated objects (manifolds), hence is \(0\)-truncated. 
Since \(0\)-truncatedness of a morphism is local on the target for effective epimorphisms, it follows that \(\Delta_X\) is \(0\)-truncated. Therefore \(X\) is \(1\)-truncated.

Finally, since \(p\) is an effective epimorphism, \(X\) is the geometric realization of its \v{C}ech nerve \(M^{\bullet}=M^{\times_X(\bullet+1)}\), and each \(M^{n}\) is a manifold by representability. 
As \(X\) is \(1\)-truncated, the simplicial object \(M^{\bullet}\) is equivalent to the nerve of the groupoid object \(R:=M\times_X M\rightrightarrows M\) in \(\Man_F\).
All structure maps are morphisms of \(F\)-analytic manifolds hence smooth, and the source and target maps \(s,t\colon R\to M\) are submersions, since they are pullbacks of \(p\).
\end{proof}
\subsection{Profinite Groupoids and Van Dantzig Theorem}\label{sec:van}
We will begin with some topological preliminaries:
\begin{definition}
    A topological space $X$ is called \emph{locally profinite} if every point has a profinite neighborhood.
\end{definition}

\begin{example}
    A $p$-adic manifold is a locally profinite space.
\end{example}

\begin{theorem}
    In a locally profinite space, every point has a neighborhood basis consisting of compact open sets.
\end{theorem}

\begin{proof}
    Let $X$ be a locally profinite space, let $x \in X$, and let $U$ be an arbitrary open neighborhood of $x$. We wish to find a compact open neighborhood $K$ of $x$ such that $K \subseteq U$.

    By definition, $x$ has a neighborhood $N$ that is a profinite space. Because $N$ is a neighborhood of $x$, there must exist an open set $V$ in $X$ such that $x \in V \subseteq N$.

    Let $W = U \cap V$. Because both $U$ and $V$ are open sets in $X$, $W$ is an open neighborhood of $x$ in $X$. Furthermore, since $W \subseteq N$, $W$ is also open in the subspace topology of $N$.

    Recall that a profinite space has a basis of clopen (closed and open) sets. Because $W$ is an open set in $N$ and contains $x$, there exists a subset $K$ of $N$ that is clopen in $N$ such that $x \in K \subseteq W$.

    We now verify that $K$ is both compact and open in the global space $X$:
    \begin{itemize}
        \item \textbf{Compactness:} Because $K$ is a closed subset of the compact Hausdorff space $N$, $K$ is compact.
        
        \item \textbf{Openness in $X$:} Because $K$ is open in the subspace topology of $N$, there exists an open set $O$ in $X$ such that $K = O \cap N$. Since we know that $K \subseteq W \subseteq V$, we can intersect both sides of our equation with $V$:
        \[
            K = K \cap V = (O \cap N) \cap V = O \cap (N \cap V)
        \]
        Since $V \subseteq N$, we know that $N \cap V = V$. Substituting this back yields $K = O \cap V$. Because both $O$ and $V$ are open sets in $X$, their intersection $K$ is open in $X$.
    \end{itemize}

    We have found a compact open set $K$ in $X$ satisfying $x \in K \subseteq U$. Since $U$ was an arbitrary open neighborhood of $x$, this proves that the compact open sets form a neighborhood basis for every point in $X$.
\end{proof}

We now state a groupoid-level generalization of van Dantzig’s theorem: under the openness of the source and target maps, open neighborhoods of the units contain wide Hausdorfff compact open subgroupoids, giving a basis of “compact–open” structure near $G_0$.
This theorem is crucial because it allows us to reduce the conjecture to the compact Hausdorff case.
\begin{theorem}\label{VanDan}
Let $G \rightrightarrows G_0$ be a groupoid object in $\Pyk(\mathcal S)$ such that the source and target maps
\[
s,t: G_1\to G_0
\]
are \emph{open maps}.
Assume $G_1$ is locally profinite and $G_0$ is profinite.
Then for every \emph{open} neighborhood $U$ of $G_0$ in $G_1$, there exists a \emph{Hausdorff compact open} subgroupoid
$H \subseteq G$ such that
\[
G_0 \subseteq H_1 \subseteq U.
\]
(In particular, $H$ is wide in the sense that it contains all units.)
\end{theorem}

\begin{proof}
    Let $U\subseteq G_1$ be an arbitrary open neighborhood of $G_0$. We first claim that there exists a \emph{compact open Hausdorff} set $K$ such that
    \[
        G_0 \subseteq K \subseteq U.
    \]
    
    To see this, note that by \cite[Proposition 2.18]{loc}, there exists an open Hausdorff neighborhood of $G_0$ in $G_1$; by replacing $U$ with its intersection with this neighborhood, we may assume without loss of generality that $U$ itself is a Hausdorff space.
    
    Since $G_1$ is locally profinite, from the last lemma it has a neighborhood basis consisting of compact open sets. Thus, for each $x \in G_0$, we can choose a compact open neighborhood $U_x$ such that $x \in U_x \subseteq U$. 
    
    The collection $\{U_x\}_{x \in G_0}$ forms an open cover of $G_0$. Because $G_0$ is a profinite space, it is compact, which allows us to extract a finite subcover $\{U_{x_1}, \dots, U_{x_n}\}$. We define $K$ to be their union:
    \[
        K = \bigcup_{i=1}^n U_{x_i}.
    \]
    
    As a finite union of compact open sets, $K$ is both compact and open. Furthermore, because $K \subseteq U$ and $U$ is a Hausdorff space, $K$ inherits the Hausdorff property. Thus, $K$ is a compact, open, Hausdorff set satisfying $G_0 \subseteq K \subseteq U$.

    We will now refine $K$ to find a compact open set $F \subseteq K$ (which will hence also be Hausdorff) such that $m(F \times_{G_0} F) \subseteq K$, where $m : G_1 \times_{G_0} G_1 \to G_1$ is the continuous multiplication map.
    
    Because $G_0$ is a profinite space, it is Hausdorff, which implies its diagonal $\Delta_{G_0} \subseteq G_0 \times G_0$ is closed. The fiber product $K \times_{G_0} K$ is exactly the preimage of $\Delta_{G_0}$ under the continuous map $t \times s: K \times K \to G_0 \times G_0$. Therefore, $K \times_{G_0} K$ is a closed subspace of $K \times K$. Since $K$ is compact Hausdorff, the product $K \times K$ is compact Hausdorff, making $K \times_{G_0} K$ a compact Hausdorff space.

    Because $K$ is an open set in $G_1$, its preimage $m^{-1}(K)$ is open in $G_1 \times_{G_0} G_1$. Consequently, the intersection 
    \[
        W = m^{-1}(K) \cap (K \times_{G_0} K)
    \]
    is an open subset of the space $G_1\times_{G_0} G_1$. Notice that for any unit $x \in G_0$, its product is simply $m(x,x) = x \in K$. This guarantees that the space of composable units, $G_0 \times_{G_0} G_0$, is entirely contained within $W$.

    Note that $K$ is itself a profinite space. As an open subspace of a locally profinite space, it inherits a basis of compact open sets(which exist from the last lemma). Because $K$ is Hausdorff, these compact sets are also closed. Thus, $K$ is a compact Hausdorff space with a basis of clopen sets, meaning it is totally disconnected and therefore profinite. In a profinite space, any closed set is exactly equal to the intersection of all clopen neighborhoods containing it.
    
    Let $\mathcal{N}$ be the downward-directed family of all clopen neighborhoods of $G_0$ in $K$. We thus have $\bigcap_{V \in \mathcal{N}} V = G_0$.
    
    For each $V \in \mathcal{N}$, \(V \times_{G_0} V\) is closed in \(K\times_{G_0} K\).

    Remark that
\[
\bigcap_{V \in \mathcal{N}} \bigl(V \times_{G_0} V\bigr)
= G_0 \times_{G_0} G_0 \subseteq W,
\]
define
\[
Z_V := \bigl(V \times_{G_0} V\bigr)  \setminus W.
\]
Then \(Z_V\) is closed in \(K\times_{G_0}K\), and
\[
\bigcap_{V \in \mathcal{N}} Z_V
=
\left(\bigcap_{V \in \mathcal{N}} (V \times_{G_0} V)\right)\cap ((K \times_{G_0} K)\setminus W)
\subseteq W \cap ((K \times_{G_0} K)\setminus W)=\varnothing.
\]
By compactness of \(K \times_{G_0} K\), there exist \(F \in \mathcal{N}\) such that
\[
Z_F =\emptyset
\]    
Hence $F \times_{G_0} F \subseteq W$. Consequently, $m(F \times_{G_0} F) \subseteq K$.

    By replacing $F$ with $F \cap F^{-1}$, we may assume without loss of generality that $F$ is a symmetric set ($F = F^{-1}$). Notice that $F$ is by construction \emph{clopen} in $K$ and hence also compact.

    Define the composable-pairs space:
    \[
        F_2 := F\times_{G_0}F := \{(g,y)\in F\times F \mid t(g)=s(y)\},
    \]
    and the set of pairs whose product escapes $F$:
    \[
        B := \{(g,y)\in F_2 \mid gy\notin F\}.
    \]
    Because $m(F_2) \subseteq K$, we can equivalently write this as the preimage $B = m|_{F_2}^{-1}(K \setminus F)$. Since $F$ is clopen in $K$, the complement $K \setminus F$ is clopen in $K$, making $B$ a clopen subset of $F_2$. As $F$ and $G_0$ are compact Hausdorff spaces, $F_2$ is compact Hausdorff, meaning $B$ is compact as a closed subspace of a compact space.

    Let $p_1:F_2\to F$ be the first projection, and set
    \[
        M := p_1(B),\qquad V := M\cup M^{-1},\qquad H_1 := F\setminus V.
    \]
    We claim that $M$ is clopen in $F$. First, $M$ is compact as the continuous image of the compact set $B$, and hence closed in the Hausdorff space $F$. 
    To see that $M$ is open, it suffices to show that $p_1$ is an open map. A projection from a fiber product $X\times_S Y \to X$ is an open map if the underlying morphism $Y \to S$ is open. Here, $p_1$ is exactly the base change (pullback) of the source map $s|_F : F \to G_0$ along the target map $t|_F : F \to G_0$. Because $F$ is an open set in $G_1$ and $s : G_1 \to G_0$ is an open map, the restriction $s|_F$ is an open map. The pullback of an open map is always open; thus, $p_1$ is an open map, making $M$ an open set.
    
    Because $M$ is clopen in $F$ and inversion is a homeomorphism, $M^{-1}$ is also clopen, meaning $V$ is clopen in $F$. Consequently, $H_1 = F \setminus V$ is clopen in $F$ (and hence open in $G_1$). Since $F$ is compact Hausdorff, $H_1$ is a compact, open, Hausdorff subset of $G_1$.

    Finally, we verify that $H_1$ forms a subgroupoid containing $G_0$.
    \begin{itemize}
        \item \textbf{Contains Units:} Because $G_0$ consists of the units, for any $x \in G_0$ and $y \in F$ composable via $t(x)=s(y)$, we have $xy = y \in F$, meaning $(x,y) \notin B$. Thus $x \notin M$. Because units are self-inverse, $x^{-1} \notin M$ either. Hence $G_0 \cap (M \cup M^{-1}) = \emptyset$, guaranteeing $G_0 \subseteq H_1$.
        
        \item \textbf{Closed under Inversion:} If $g \in H_1$, then $g \in F$ and $g \notin M \cup M^{-1}$. Because we ensured $F$ is symmetric, $g^{-1} \in F$. Furthermore, $g^{-1} \notin M^{-1}$ and $(g^{-1})^{-1} = g \notin M$. Thus $g^{-1} \in H_1$, proving closure under inversion.
        
        \item \textbf{Closed under Composition:} If $g,h \in H_1$ are composable (meaning $t(g)=s(h)$), then $g \notin M$ and $h \in F$ guarantees their product $gh \in F$. Assume for the sake of contradiction that $gh \in M$. Then there exists $y \in F$ with $t(gh) = s(y)$ such that $(gh)y \notin F$. 
        Because $t(gh)=t(h)$, we have $t(h)=s(y)$, meaning $(h,y) \in F_2$. Because $h \in H_1$, we know $h \notin M$, which forces $hy \in F$.
        
        Now consider the pair $(g, hy)$. Because $s(hy) = s(h) = t(g)$, they are composable, meaning $(g, hy) \in F_2$. Because $g \notin M$, their product $g(hy)$ must remain in $F$. However, by associativity, $g(hy) = (gh)y \notin F$, which is a contradiction.
        Thus $gh \notin M$. 
        
        By applying the exact same logic to their inverses (since $g^{-1}, h^{-1} \in H_1$ are composable via $t(h^{-1})=s(g^{-1})$, we similarly get $h^{-1}g^{-1} \notin M$), we conclude $(gh)^{-1} \notin M$, meaning $gh \notin M^{-1}$. Therefore, $gh \notin M \cup M^{-1}$, so $gh \in H_1$.
    \end{itemize}
    
    Since $H_1 \subseteq F \subseteq U$ and $H_1$ defines a compact, open, Hausdorff subgroupoid containing $G_0$, this completes the proof.
\end{proof}

We will need the following key local property: passing to a compact clopen wide subgroupoid yields an \'etale map on realizations.

\begin{lemma}[\'Etaleness of clopen subgroupoid quotients]
\label{lem:etale-clopen-subgroupoid}
Let $G \rightrightarrows G_0$ be a p-adic lie groupoid.
Then for every open wide subgroupoid $H \subseteq G$, the induced map on geometric realizations
\[
|H| \longrightarrow |G|
\]
is \'etale.
\end{lemma}
\begin{proof}
Let $F:=|G|$ and $F':=|H|$. Let $G_0$ and consider the object space $G_0$ and write $p:G_0\twoheadrightarrow F$ for the canonical atlas.
Since $p$ is an effective epimorphism, by Lemma~\ref{etaleprop}(4) it suffices to show that the base change
\[
P \;:=\; F'\times_F G_0 \longrightarrow G_0
\]
is \'etale.

\smallskip
\noindent\emph{Description of $P$.}
Base changing along $G_0\to F=|G|$ identifies $P$ with:
$$ P \cong G_1 \big/ H_1 $$ where $H_1$ acts freely on the space of arrows $G_1$ by left composition: $h \cdot g = h \circ g$ (which is well-defined since $s(h) = t(g) \in H_0$). The map $P \to G_0$ is induced by the source map $s: G_1 \to G_0$.
Define the division map 
\[
\delta \colon G_1 \times_t G_1 \longrightarrow G_1, \qquad (g_1, g_2) \longmapsto g_2^{-1}g_1.
\]
Two elements $g_1, g_2 \in G_1$ belong to the same right $H$-coset if and only if $t(g_1) = t(g_2)$ and $g_2^{-1}g_1 \in H$. Therefore, the equivalence relation is exactly the preimage $R = \delta^{-1}(H)$ since $H$ is open $\delta^{-1}(H)$ is open too and hence $R$ is an embedded manifold moreover the projection is a submersion, hence by the Godement criterion \cite[Part II, Ch. III, \S 12.2, p. 92]{serre}, the quotient $G_1/H$ admits a unique $F$-analytic manifold structure such that $q$ is a submersion.
 Finally, since the map $s:G_1\to G_0$ factors through the structure map $P \to G_0$ and $G_1\to P$ is a surjective submersion we conclude that $P\to G_0$ is a submersion too.

\smallskip
\noindent\emph{Discrete fibers.}
The equivalence classes on $G_1$ are the orbits of the $H_1$ action. Fix an $s$-fiber $s^{-1}(x)$. For any $g \in s^{-1}(x)$, left-composition by $g$ yields a diffeomorphism:$$ R_g: s^{-1}(t(g)) \to s^{-1}(x), \quad h \mapsto h \circ g $$Because $H_1$ is a open subspace of $G_1$, its intersection with the fiber $s^{-1}(t(g))$ is open. Its image under $R_g$ is exactly the orbit $H_1 \circ g$. Thus, the $H_1$-orbits are open, hence $P\to G_0$ has discrete fibers.

\smallskip
\noindent\emph{Conclusion.}
A submersion with discrete fibers is a local homeomorphism, hence \'etale in $\mathrm{CondAn}^{\mathrm{light}}$
(cf.\ \cite[Lemma~4.20.2]{Clausen}). Therefore $P\to G_0$ is \'etale, and consequently the map
$|H|=F'\to F=|G|$ is \'etale.
\end{proof}
The rest of this section is devoted to understanding the compact Hausdorff subgroupoids obtained by applying Theorem~\ref{VanDan} to a $p$-adic Lie group. Hence, the object and morphism spaces of the resulting subgroupoid are light profinite.

To simplify the combinatorics of $\mathcal G$ without changing its realization, we pass to a skeletal replacement; this in turn makes it possible to define the normal core of a subgroupoid.
\begin{lemma}[Skeletal replacement]
\label{lem:skeletal-replacement}
Let $\mathcal{G}$ be a groupoid object whose object and morphism spaces are light profinite, and whose source and target maps are open.
Then there exists a groupoid $\mathcal{X}$ and a functor $p:\mathcal{X}\to \mathcal{G}$ such that:
\begin{enumerate}[label=(\roman*),leftmargin=2.2em]
    \item $\Obj(\mathcal{X})$ and $\Mor(\mathcal{X})$ are light profinite;
    \item the induced map on geometric realizations
    \[
    |p|:\,|\mathcal{X}|\longrightarrow |\mathcal{G}|
    \]
    is an equivalence;
    \item $\mathcal{X}$ is skeletal in the sense that for any two distinct objects
    $x\neq y$ in $\Obj(\mathcal{X})$ one has $\Hom_{\mathcal{X}}(x,y)=\varnothing$.
\end{enumerate}
\end{lemma}

\begin{proof}
Consider the quotient map $q:\Obj(\mathcal{G})\to \pi_0(\mathcal{G})$.
This map is open since $s$ and $t$ are open. Furthermore, since the orbit equivalence relation is closed, $\pi_0(\mathcal{G})$ is a compact Hausdorff space. Since $\Obj(\mathcal{G})$ is a totally disconnected compact space, the open continuous image of a compact clopen basis set is compact and open, hence clopen in the Hausdorff quotient $\pi_0(\mathcal{G})$. Thus $\pi_0(\mathcal{G})$ admits a basis of clopen sets, meaning it is a totally disconnected compact space.

Moreover, $q$ is an open surjection of light profinite sets.
Although open surjections of arbitrary profinite spaces do not necessarily admit continuous sections \cite{erratum}, $q$ is an open surjection of light profinite sets (which correspond to countable inverse limits), and therefore by \cite{erratum} a continuous section of $q$, $\sigma:\pi_0(\mathcal{G})\to \Obj(\mathcal{G})$ exists.

Define $\mathcal{X}$ as the following subgroupoid of $\mathcal{G}$:
\[
\Obj(\mathcal{X})=\sigma\bigl(\pi_0(\mathcal{G})\bigr),\qquad
\Mor(\mathcal{X})=s^{-1}\!\bigl(\sigma(\pi_0(\mathcal{G}))\bigr)\cap t^{-1}\!\bigl(\sigma(\pi_0(\mathcal{G}))\bigr).
\]
Let $p:\mathcal{X}\hookrightarrow \mathcal{G}$ be the evident inclusion functor.

Then (i) holds since $\Obj(\mathcal{X})$ and $\Mor(\mathcal{X})$ are closed subspaces of light profinite spaces.
Condition (iii) is clear by construction: $\Obj(\mathcal{X})$ meets each $\pi_0$-class in exactly one point, so there are
no morphisms between distinct objects of $\mathcal{X}$.
Finally, (ii) follows from Theorem~\ref{functor_theorem}.
\end{proof}
\begin{remark}
    The previous lemma says that a compact groupoid satisfying the condition of the lemma is equivalent to a group object in the slice topos
\end{remark}
The next theorem formalizes the key consequence of skeletality: every clopen neighborhood of $\mathcal X_0$ contains a clopen wide \emph{normal} subgroupoid.
\begin{theorem}
\label{thm:normal-basis-skeleton}
Let $\mathcal{G}$ be a groupoid object whose object and morphism spaces are light profinite, and whose source and target maps are open and let $\mathcal{X}$ be the skeleton of $\mathcal{G}$ from Lemma~\ref{lem:skeletal-replacement}. Then $\mathcal{X}$ admits a neighborhood basis
of $\mathcal{X}_0$ in $\mathcal{X}_1$ consisting of clopen wide normal subgroupoids. Equivalently, for every clopen neighborhood
$U\subseteq \mathcal{X}_1$ of $\mathcal{X}_0$, there exists a clopen wide normal subgroupoid
$\mathcal{N}\subseteq \mathcal{X}$ such that
\[
\mathcal{X}_0\subseteq \mathcal{N}_1\subseteq U.
\]
\end{theorem}

\begin{proof}
Intersecting with $\mathcal{X}$ the basis given by Theorem~\ref{VanDan} gives a basis
\[
\mathcal{B}=\{H\cap \mathcal{X}\}
\]
of comapct open wide subgroupoids of $\mathcal{X}$ which are actually clopen because $\mathcal{X}_1$ is Hausdorff.
We refine this basis to a basis by normal subgroupoids.

Fix $H':=H\cap \mathcal{X}\in \mathcal{B}$.
We will define a subset $H_{\mathrm{bad}}\subseteq \Mor(H)$ such that
\[
H'_{\mathrm{bad}}:=H_{\mathrm{bad}}\cap \Mor(\mathcal{X})
\subseteq \Mor(H)\cap \Mor(\mathcal{X})=H'_1
\]
is clopen in $\Mor(\mathcal{X})$, and such that
\[
\mathrm{Core}(H'):=H'_1\setminus H'_{\mathrm{bad}}
\]
is the morphism space of a clopen wide normal subgroupoid of $\mathcal{X}$ contained in $H'$.

Let $\Iso(H)\subseteq \Mor(H)$ denote the isotropy (loops), and consider the conjugation action map
\[
c:\Iso(H)\times_{\,\Obj(\mathcal{G})}\Mor(\mathcal{G})\longrightarrow \Mor(\mathcal{G}),
\qquad (r,g)\longmapsto grg^{-1},
\]
(where the fiber product encodes the usual composability condition for conjugation).

Let $C\subseteq \Iso(H)\times_{\Obj(\mathcal{G})}\Mor(\mathcal{G})$ be the inverse image of the complement:
\[
C:=c^{-1}\!\bigl(\Mor(\mathcal{G})\setminus \Mor(H)\bigr).
\]
Since $\Mor(\mathcal{G})\setminus \Mor(H)$ is clopen and $c$ is continuous, the set $C$ is clopen.
Moreover $C$ is compact (it is closed in a compact space).

Let $p_1:\Iso(H)\times_{\Obj(\mathcal{G})}\Mor(\mathcal{G})\to \Iso(H)$ be the projection and define
\[
H_{\mathrm{bad}}:=p_1(C)\subseteq \Iso(H)\subseteq \Mor(H).
\]
Then $H_{\mathrm{bad}}$ is compact as the image of a compact set, hence closed in $\Iso(H)$.
Furthermore, $p_1$ is an open map since it is obtained by pullback of the target map (which is open by hypothesis),
so $H_{\mathrm{bad}}$ is open as well. Hence $H_{\mathrm{bad}}$ is clopen in $\Iso(H)$.

Finally set $H'_{\mathrm{bad}}:=H_{\mathrm{bad}}\cap H'_1$.
Since $H'_1$ is a subspace of $\Iso(H)$, $H'_{\mathrm{bad}}$ is clopen in $H'_1$ hence clopen in $\mathcal{X}$ (here we use that $H'_1$ is clopen in $\mathcal{X}$).
By construction, $\mathrm{Core}(H')=H'_1\setminus H'_{\mathrm{bad}}$ is stable under conjugation by arrows of $\mathcal{X}$,
hence defines a clopen wide normal subgroupoid of $\mathcal{X}$ contained in $H'$.
\end{proof}
Having produced a normal clopen neighborhood basis, we can now reconstruct $\mathcal X$ as the inverse limit of its normal quotients.
\begin{lemma}[Reconstruction as an inverse limit of quotients]
\label{lem:X-invlim-quotients}
Let $\mathcal{G}$ be a groupoid object whose object and morphism spaces are light profinite, and whose source and target maps are open and let $\mathcal{X}$ be the skeleton of $\mathcal{G}$ from Lemma~\ref{lem:skeletal-replacement}. Let $\mathcal{B}$ be a neighborhood basis
of $\mathcal{X}^{(0)}$ in $\mathcal{X}^{(1)}$ consisting of clopen wide normal subgroupoids (as in
Theorem~\ref{thm:normal-basis-skeleton}). For $H\in\mathcal{B}$ write $\mathcal{X}/H$ for the quotient groupoid, and order $\mathcal{B}$
by reverse inclusion.

Then the canonical map of groupoid objects
\[
\mathcal{X}\longrightarrow \varprojlim_{H\in\mathcal{B}}\, \mathcal{X}/H
\]
is an isomorphism.
\end{lemma}
\begin{proof}
Since limits of simplicial objects are computed levelwise, it suffices to check that the induced maps on objects and morphisms
\[
\mathcal X_0 \longrightarrow \varprojlim_{H\in\mathcal B} (\mathcal X/H)_0,
\qquad
\mathcal X_1 \longrightarrow \varprojlim_{H\in\mathcal B} (\mathcal X/H)_1
\]
are isomorphisms.

For the object space there is nothing to prove: by definition $(\mathcal X/H)_0=\mathcal X_0$ for all $H$, and the transition
maps are identities, hence
\[
\varprojlim_{H\in\mathcal B} (\mathcal X/H)_0 \cong \mathcal X_0.
\]

For morphisms, write $q_H:\mathcal X_1\to \mathcal X_1/H_1=(\mathcal X/H)_1$ for the quotient map. Since $\mathcal B$ is ordered
by reverse inclusion, the maps $q_H$ assemble into a canonical map
\[
q:\mathcal X_1 \longrightarrow \varprojlim_{H\in\mathcal B} \mathcal X_1/H_1.
\]
We claim that $q$ is a homeomorphism.

\smallskip\noindent
\emph{Injectivity.}
If $g\neq g'$ in $\mathcal X_1$, then (as $\mathcal X_1$ is Hausdorff and totally disconnected) there exists a clopen neighborhood
$U\subseteq \mathcal X_1$ of $\mathcal X_0$ such that $g^{-1}g'\notin U$. Choose $H\in\mathcal B$ with $H_1\subseteq U$.
Then $g^{-1}g'\notin H_1$, hence $q_H(g)\neq q_H(g')$, so $q$ is injective.

\smallskip\noindent
\emph{Surjectivity.}
Let $(\bar g_H)_{H\in\mathcal B}$ be a compatible family with
$\bar g_H\in \mathcal X_1/H_1$.
Choose representatives $g_H\in \mathcal X_1$ with $q_H(g_H)=\bar g_H$.
For each $H\in\mathcal B$ set
\[
C_H := g_H H_1 \subseteq \mathcal X_1.
\]
Each $C_H$
is a nonempty closed subset of $\mathcal X_1$.

We claim that the family $(C_H)_{H\in\mathcal B}$ has the finite intersection property.
Indeed, given $H_1,\dots,H_n\in\mathcal B$, choose $K\in\mathcal B$ with
$K\subseteq H_i$ for all $i$ (cofilteredness).
Compatibility implies $q_{H_i}(g_K)=q_{H_i}(g_{H_i})$, hence
$g_K\in g_{H_i} (H_i)_1 = C_{H_i}$ for each $i$.
Thus $g_K\in \bigcap_{i=1}^n C_{H_i}$, so this intersection is nonempty.

Since $\mathcal X_1$ is compact, we conclude that
\[
\bigcap_{H\in\mathcal B} C_H \neq \varnothing.
\]
Pick $g$ in this intersection. Then for every $H\in\mathcal B$ we have
$g\in g_H H_1$, hence $q_H(g)=q_H(g_H)=\bar g_H$.
Therefore $q(g)=(\bar g_H)_H$, proving surjectivity.
\smallskip\noindent

The map $q$ is continuous by construction. Since $\mathcal X_1$ is compact and the limit $\varprojlim_H \mathcal X_1/H_1$ is
Hausdorff (inverse limit of compact Hausdorff spaces), a continuous bijection is a homeomorphism. Hence
\[
\mathcal X_1 \xrightarrow{\ \sim\ } \varprojlim_{H\in\mathcal B} (\mathcal X/H)_1.
\]

\end{proof}

To prove that $\mathcal X/\mathcal N$ is an inverse limit of groupoids with finite spaces of objects and morphisms, we will need a way to collapse a quotient groupoid onto a specified connected component; the following lemma provides such a section.
\begin{lemma}[Collapse Maps]\label{collapse}
Let $\mathcal{G}$ be a groupoid object whose object and morphism spaces are light profinite, and whose source and target maps are open and let $\mathcal{X}$ be the skeleton of $\mathcal{G}$ from Lemma~\ref{lem:skeletal-replacement}. 
Let $\mathcal{N}$ be a clopen normal subgroupoid of $\mathcal{X}$, and let $E = \mathcal{X}/\mathcal{N}$. For any basepoint $x \in \mathcal{X}_0$, there exists a collapse map $\Phi: E \to  E_x$ which is a section of the inclusion $E_x\hookrightarrow E$.
\end{lemma}
\begin{proof}
Fix $x \in E_0$. Observe that since $\mathcal{N}_x$ is an open subgroup of the compact group $\mathcal{X}_x$, the quotient 
\[ E_x=\mathcal{X}_x/\mathcal{N}_x \] 
is finite.
Moreover, $E_1$ is a compact Hausdorff space as a closed quotient of $\mathcal{X}_1$.
Consider the map $E_1\to \mathcal{X}_0$; we claim that $E_1$ is a profinite space. Since it is compact Hausdorff, it suffices to prove it is totally disconnected.
To see that $E_1$ is totally disconnected, let $C$ be a connected component of $E_1$. Its image under the source map is a connected component of the totally disconnected space $\mathcal{X}_0$, hence a single point. Therefore, $C$ belongs entirely to a single discrete fiber $E_y$. Since the fibers are finite (hence discrete), we conclude $C$ must be a single point. Thus, $E_1$ is totally disconnected.

The fiber is a finite group $E_x = \{g_1, \dots, g_m\}$, where $g_1 = e_x$. Because $E_1$ is totally disconnected, we can isolate each $g_i$ in disjoint clopen sets $V_i \subset E_1$. By continuity of multiplication, there exist smaller clopen neighborhoods $W_i \subset V_i$ of $g_i$ satisfying $W_i W_j \subseteq V_k$ (where $g_k = g_i g_j$). Let $W = \bigcup_{i=1}^m W_i$. 

The set $Z = E_1 \setminus W$ is the closed set of arrows avoiding our tubes. Since $E_1$ is compact and $E_0$ is Hausdorff, the source map $s$ is closed. Thus, $s(Z)$ is a closed set in $E_0$. Since the entire fiber over $x$ is contained in $W$, we have $x \notin s(Z)$. The complement $E_0 \setminus s(Z)$ is an open neighborhood of $x$. Since $E_0$ is a Stone space, we may choose a basic clopen neighborhood $U$ such that $x \in U \subseteq E_0 \setminus s(Z)$ shrink $U$ further so that $U \subseteq u^{-1}(W_1)$. 

By construction, $s^{-1}(U) \subseteq W$. If we define the restricted tubes $W'_i = W_i \cap s^{-1}(U)$, we obtain a perfect clopen partition:
\[
s^{-1}(U) = W'_1 \sqcup W'_2 \sqcup \dots \sqcup W'_m.
\]
Crucially, this tightens the multiplication: if $a \in W'_i$ and $b \in W'_j$ are composable, their product $a \circ b \in V_k$. Since $\mathcal{X}$ (and thus $E$) is skeletal, $a$ and $b$ must originate at the same object $y \in U$. Thus, $a \circ b \in s^{-1}(U) \subseteq \bigcup W_r \subseteq \bigcup V_r$. Because the $V_r$ are mutually disjoint, we are forced to conclude $a \circ b \in W'_k$.

We define a global map $\Phi: E \to E_x$ piecewise:
\begin{itemize}
    \item For $h \in s^{-1}(U)$, $h$ belongs to a unique $W'_i$, and we set $\Phi(h) = g_i$.
    \item For $h \in s^{-1}(U^c)$, we set $\Phi(h) = e_x$.
\end{itemize}
Because $\mathcal{X}$ is skeletal, there are no composable arrows crossing between $U$ and $U^c$. The multiplication preserves the piecewise components, and our refinement $W'_i W'_j \subseteq W'_k$ guarantees $\Phi$ preserves multiplication on $U$. Thus $\Phi$ preserves composition globally. Because $U$, $U^c$, and $W'_i$ are all clopen, $\Phi$ is continuous.
\end{proof}

We will now prove that the quotient groupoid $E = \mathcal{X}/\mathcal{N}$ is an inverse limit of groupoids such that both the object and morphism sets are finite.

\begin{proposition}[$E = \mathcal{X}/\mathcal{N}$ is a profinite groupoid]\label{prop:E-profinite-groupoid}
Let $\mathcal{G}$ be a groupoid object whose object and morphism spaces are light profinite, and whose source and target maps are open and let $\mathcal{X}$ be the skeleton of $\mathcal{G}$ from Lemma~\ref{lem:skeletal-replacement}. 
Let $\mathcal{N}$ be a clopen normal subgroupoid of $\mathcal{X}$, then the quotient groupoid $E = \mathcal{X}/\mathcal{N}$ is isomorphic to an inverse limit of finite discrete groupoids (i.e., groupoids with finite sets of objects and morphisms).
\end{proposition}

\begin{proof}
Let $\mathcal{I}$ be the category of continuous groupoid homomorphisms $\Psi \colon E \to F$, where $F$ is a finite discrete intransitive groupoid. Because the category of finite intransitive groupoids is closed under finite limits, the comma category $\mathcal{I}$ of finite quotients of $E$ is cofiltered. We must show that the canonical evaluation functor
\[
\Theta \colon E \longrightarrow \varprojlim_{(F, \Psi) \in \mathcal{I}} F
\]
is an isomorphism of topological groupoids.

By Lemma~\ref{collapse}, both $E_0$ and $E_1$ are compact Hausdorff totally disconnected spaces (Stone spaces). A continuous bijection from a compact space to a Hausdorff space is automatically a homeomorphism. Thus, it suffices to prove that continuous functors to finite groupoids separate points in both $E_0$ and $E_1$.

\textbf{Separating objects:} Let $y, z \in E_0$ be distinct objects. Because $E_0$ is totally disconnected, there exists a clopen partition $E_0 = U \sqcup V$ such that $y \in U$ and $z \in V$.
Let $F$ be the finite discrete groupoid with two objects $\{1, 2\}$ and only identity morphisms. Because $E$ is skeletal, there are no morphisms crossing between $U$ and $V$. Thus, we can legitimately define a continuous functor $\Psi \colon E \to F$ mapping all objects and morphisms in $E|_U$ to the object $1$ and its identity $e_1$, and those in $E|_V$ to the object $2$ and its identity $e_2$. We have $\Psi(y) = 1 \neq 2 = \Psi(z)$, so $y$ and $z$ are strictly separated.

\textbf{Separating morphisms:} Let $a, b \in E_1$ be distinct morphisms.
\begin{itemize}
    \item \emph{Case 1: $s(a) \neq s(b)$.} We map $E$ to the finite groupoid $F$ separating the source objects $s(a)$ and $s(b)$ as constructed above. Since $\Psi(a) = e_{\Psi(s(a))}$ and $\Psi(b) = e_{\Psi(s(b))}$, and $\Psi(s(a)) \neq \Psi(s(b))$, it follows that $\Psi(a) \neq \Psi(b)$.

    \item \emph{Case 2: $s(a) = s(b) = x$.} In this case, because $E$ is skeletal, both arrows are endomorphisms of $x$, meaning $a, b \in E_x$. 
    By Lemma~\ref{collapse} (Collapse Maps), there exists a continuous global collapse functor $\Phi \colon E \to E_x$. 
    Here, we view the finite group $E_x$ as a finite discrete groupoid with a single object. Because $\Phi$ is a section of the inclusion $E_x \hookrightarrow E$, it acts as the exact identity on $E_x$. Thus, $\Phi(a) = a$ and $\Phi(b) = b$. Because $a \neq b$ inside the finite group $E_x$, the continuous functor $\Phi$ to the finite groupoid $E_x$ strictly separates $a$ and $b$.
\end{itemize}

Since continuous functors to finite groupoids separate all objects and morphisms, the canonical map $\Theta$ is an injective continuous groupoid homomorphism, and is therefore a topological isomorphism onto its image. 
\smallskip\noindent
\\
\emph{Surjectivity.} We prove surjectivity uniformly for the spaces of objects and morphisms. Let $k \in \{0, 1\}$, so $E_k$ is either the object or morphism space of $E$.
Take a compatible point $(c_\Psi)_{\Psi \in \mathcal{I}}$ in the inverse limit $\varprojlim_{\Psi \in \mathcal{I}} F_k$. For each continuous functor $\Psi \colon E \to F$ in $\mathcal{I}$, define the preimage
\[
C_\Psi := \Psi_k^{-1}(c_\Psi) \subseteq E_k.
\]
Because $F_k$ is a discrete space and $\Psi_k$ is continuous, $C_\Psi$ is a clopen (hence closed) subset of $E_k$.

First, we verify that $C_\Psi$ is nonempty. Because $E$ is skeletal, its set-theoretic image $\mathrm{Im}(\Psi)$ is a finite skeletal subgroupoid of $F$. Thus, the corestriction $\Psi' \colon E \twoheadrightarrow \mathrm{Im}(\Psi)$ is an object in $\mathcal{I}$, and the inclusion $\iota \colon \mathrm{Im}(\Psi) \hookrightarrow F$ is a transition morphism. Compatibility in the limit says $c_\Psi = \iota_k(c_{\Psi'}) = c_{\Psi'}$, proving $c_\Psi$ actually lies in the set-theoretic image of $\Psi_k$. Therefore, $C_\Psi \neq \emptyset$.

Compatibility in the limit exactly says that if $\Phi \succeq \Psi$, meaning there is a transition morphism $u \colon G \to F$ such that $u \circ \Phi = \Psi$, then $u_k(c_\Phi) = c_\Psi$. Consequently, if $x \in C_\Phi$, we have $\Psi_k(x) = u_k(\Phi_k(x)) = u_k(c_\Phi) = c_\Psi$, which means $x \in C_\Psi$. Thus:
\[
C_\Phi \subseteq C_\Psi.
\]
Because the index category $\mathcal{I}$ is cofiltered, any finite collection of functors $\Psi_1, \dots, \Psi_n$ is dominated by a common refinement $\Phi$. Therefore, $C_\Phi \subseteq \bigcap_{i=1}^n C_{\Psi_i}$, meaning every finite intersection of the sets $C_\Psi$ is nonempty. This shows that $\{C_\Psi\}_{\Psi \in \mathcal{I}}$ is a family of nonempty closed subsets possessing the finite intersection property.

Because $E_k$ is compact, the total intersection is nonempty:
\[
\bigcap_{\Psi \in \mathcal{I}} C_\Psi \neq \emptyset.
\]
Pick an element $x$ in this intersection. Then for every $\Psi \in \mathcal{I}$, we have $x \in C_\Psi$, i.e., $\Psi_k(x) = c_\Psi$. Thus, $\Theta_k(x) = (c_\Psi)_{\Psi \in \mathcal{I}}$, proving that $\Theta$ is surjective on both objects and morphisms.
\end{proof}
Putting the previous lemmas together yields the desired profinite presentation of $\mathcal X$.
\begin{corollary}\label{prof}
    Let $\mathcal{G}$ be a groupoid object whose object and morphism spaces are light profinite, and whose source and target maps are open and let $\mathcal{X}$ be the skeleton of $\mathcal{G}$ from Lemma~\ref{lem:skeletal-replacement}. , then $\mathcal{X}$ is an inverse limit of finite groupoid
\end{corollary}
\begin{proof}
    From \ref{lem:X-invlim-quotients}, $\mathcal{X}$ is the inverse limit of the quotient $\mathcal{X}/\mathcal{N}$.
    Moreover from the last lemma each $\mathcal{X}/\mathcal{N}$ is an inverse limit of finite groupoids.
\end{proof}

\subsection{Profiniteness\label{lem:finite-approx-quotient}}
We now translate the inverse-limit description of $\mathcal X$ into a light profiniteness statement for its geometric realization.
\begin{theorem}\label{thm:X-to-Xpi-equivalence}
Let $\mathcal{G}$ be a groupoid object whose object and morphism spaces are light profinite, and whose source and target maps are open and let $\mathcal{X}$ be the skeleton of $\mathcal{G}$ from Lemma~\ref{lem:skeletal-replacement}. 
Let $X \in \Pyk(\mathcal S)$ be the geometric realization of the groupoid
$\mathcal X$.
Then $X$ is a light profinite anima.
\end{theorem}

\begin{proof}
By Proposition~\ref{prof}, the groupoid object $\mathcal X$ can be written as a cofiltered limit
\[
\mathcal X \;\simeq\; \varprojlim_{i} F_i
\]
of finite groupoids $F_i$.
There is a canonical comparison map
\[
\alpha \colon \bigl|\mathcal X\bigr| \;=\; \Bigl|\varprojlim_i F_i\Bigr| \longrightarrow \varprojlim_i |F_i|.
\]
We claim that $\alpha$ is an equivalence. Since both source and target are coherent,
Proposition~\ref{whitehead} reduces this to checking that $\alpha$ induces an isomorphism on
$\pi_0$, and on $\pi_1$ at every basepoint.

\smallskip

\noindent\textbf{Step 1: $\pi_0$.}
Because each $F_i$ is a skeletal groupoid and $\pi_0$ commutes with cofiltered limits of
$\pi$-finite spaces \cite[Corollary~3.12]{DerivedStoneEmbedding}, we have
\[
\pi_0\!\left(\varprojlim_i |F_i|\right)
\;\simeq\; \varprojlim_i \pi_0(|F_i|)
\;\simeq\; \varprojlim_i (F_i)_0
\;\simeq\; \mathcal X_0
\;\simeq\; \pi_0\!\left(|\mathcal X|\right).
\]
Thus $\pi_0(\alpha)$ is an isomorphism.

\smallskip

\noindent\textbf{Step 2: $\pi_1$.}
Fix a point $x\in \pi_0(|\mathcal X|)\simeq \mathcal X_0$, and write $*_{x}$ for the
corresponding basepoint. Using again that $\pi_1(-,x)$ commutes with cofiltered limits of
$\pi$-finite pointed spaces \cite[Corollary~3.12]{DerivedStoneEmbedding} and that the $F$ and $E$ are skeletal, we obtain
\[
\pi_1\!\left(\varprojlim_i |F_i|,x\right)
\;\simeq\; \varprojlim_i \pi_1(|F_i|,x)
\;\simeq\; \varprojlim_i \bigl((F_i)_1 \times_{(F_i)_0} *_{x}\bigr)
\;\simeq\;(\varprojlim_i (F_i)_1) \times_{(F_i)_0} *_{x}\;\simeq\;\ \mathcal X_1 \times_{\mathcal X_0} *_{x}
\;\simeq\; \pi_1(|\mathcal X|,x),
\]
so $\pi_1(\alpha)$ is an isomorphism at every basepoint.

By Proposition~\ref{whitehead}, it follows that $\alpha$ is an equivalence, hence
\[
X \;=\; |\mathcal X| \;\simeq\; \varprojlim_i |F_i|
\]
is profinite.

Finally, $X$ is \emph{light}: indeed, it is the geometric realization of a groupoid object in
$\mathrm{CondAn}^{\mathrm{light}}$ (by construction of $\mathcal X$ in
Lemma~\ref{lem:skeletal-replacement}). Therefore Proposition~\ref{light} implies that $X$ is a
light profinite anima.
\end{proof}
\section{Proof of Conjecture 5.17}
We begin the proof by reducing \'etale-locally to the case of a compact atlas.
\begin{lemma}\label{lem:compact_atlas_etale_local}
Let $F$ be a smooth $p$-adic analytic Artin stack. Then there exists an \'etale effective epimorphism
\[
\coprod_i F_i \longrightarrow F
\]
with each $F_i$ a smooth $p$-adic analytic Artin stack, such that for every $i$ there is a surjective representable submersion
\[
M_i \longrightarrow F_i
\]
where each $M_i$ is homeomorphic to $\mathbb{Z}_p^d$.
\end{lemma}

\begin{proof}
Choose a smooth atlas of $F$ by a $p$-adic manifold, i.e.\ a surjective representable submersion
\[
p\colon M \longrightarrow F.
\]
Let $R:=M\times_F M$ with source and target maps which are submersions $s,t\colon R\rightrightarrows M$.

Since $M$ is a $p$-adic manifold, we can find a cover $\coprod M_i \to M$, such that each $M_i \to M$ is an open inclusion and each $M_i$ is homeomorphic to a $\mathbb{Z}_p^d$ for some $d$.

For each $i$, consider the composite $M_i \hookrightarrow M \xrightarrow{p} F$ and let
\[
F_i \;:=\; \bigl| \check C(M_i/F)_\bullet \bigr|
\]
be the geometric realization of its \v{C}ech nerve. Concretely, $F_i$ is presented by the simplicial object
\[
M_i \;\substack{\xleftarrow{}\\[-2pt]\xleftarrow{}}\; M_i\times_F M_i
\;\substack{\xleftarrow{}\\[-2pt]\xleftarrow{}\\[-2pt]\xleftarrow{}}\;
M_i\times_F M_i\times_F M_i \;\cdots.
\]
By construction, the natural map $M_i\to F_i$ is a surjective representable submersion, and $F_i$ are smooth.

It remains to show that $F_i\to F$ is \'etale.
Since \'etaleness is local on the target and can be tested after base change along an effective epimorphism (\ref{etaleprop}),
it suffices to check that the base change along $p$,
\[
F_i\times_F M \longrightarrow M,
\]
is \'etale.
Under this base change the map \[
F_i\times_F M \longrightarrow M,
\] is the the inclusion $t(s^{-1}(M_i))\hookrightarrow M$ which is an open immersion since $t$ is an open map.
In particular it is \'etale, so $F_i\to F$ is \'etale by \ref{etaleprop}.
\end{proof}
As the next step, we extract from the compact atlases an \'etale cover by $1$-truncated coherent objects admitting compact Hausdorff totally disconnected groupoid presentations.
\begin{lemma}\label{lem:etale_cover_1trunc_coherent}
Let \(F\) be a smooth \(p\)-adic analytic Artin stack. Then there exists an \'etale cover
\[
\coprod_{i} F_i \longrightarrow F
\]
such that each \(F_i\) is a \(1\)-truncated coherent object. Moreover, for each \(i\), the stack \(F_i\) admits a presentation as the geometric realization of a groupoid object
\[
R_i \rightrightarrows U_i
\]
whose source and target maps \(s,t\colon R_i \to U_i\) are open. In addition, \(R_i\) and \(U_i\) are compact, Hausdorff, and totally disconnected.
\end{lemma}

\begin{proof}
Since the claim is \'etale-local on $F$ by Lemma~\ref{lem:compact_atlas_etale_local} we may assume that $F$
has a cover by $\mathbb{Z}_p^d$, and as a consequence is the realization $|G|$ of a groupoid $G\rightrightarrows M$ satisfying the
hypotheses of~\ref{VanDan}. (Here we use \ref{thm:smooth-artin-stacks-are-1truncated} too) Choose a compact Hausdorff open
subgroupoid $G'\subseteq G$ as in~\ref{VanDan}, and set $F':=|G'|$.
Observe that such \(F'\) covers \(F\). Moreover, the groupoid \(G'\) has open source and target maps and is an open subgroupoid of a groupoid satisfying these conditions. Moreover, from $\ref{VanDan}$ it is compact, Hausdorff, and totally disconnected. Therefore, it suffices to show that \(F' \to F\) is \'etale which is exactly \ref{lem:etale-clopen-subgroupoid}.
\end{proof}

We can now apply the general light-profinite criterion to each $F_i$.
\begin{corollary}\label{cor:Fi_profinite_anima}
Let $F$ be one of the stacks $F_i$ from Lemma~\ref{lem:etale_cover_1trunc_coherent}. Then $F$ is light profinite.
\end{corollary}

\begin{proof}
First, by Lemma~\ref{lem:skeletal-replacement}, we see that $F_i$ satisfies the conditions of Theorem~\ref{thm:X-to-Xpi-equivalence}; hence $F_i$ are light profinite.
\end{proof}
We are now ready to conclude the proof of Conjecture~5.17.

\begin{corollary}[\label{thm:clausen-conj}{\cite[Conjecture 5.17]{Clausen}}]
Let $F$ be a smooth $p$-adic analytic Artin stack. Then \'etale-locally on $F$, the stack is a \emph{light profinite anima}.
\end{corollary}

\begin{proof}
Combine Lemma~\ref{lem:etale_cover_1trunc_coherent} with Corollary~\ref{cor:Fi_profinite_anima}.
\end{proof}

\bibliographystyle{alphaurl} 
\bibliography{meta/references} 

@book{HTT,
  title={Higher Topos Theory (AM-170)},
  author={Lurie, Jacob},
  isbn={9780691140490},
  lccn={2008038170},
  series={Annals of Mathematics Studies},
  url={https://books.google.co.il/books?id=CTe68E8wK4QC},
  year={2009},
  publisher={Princeton University Press}
}

@misc{Clausen,
  author        = {Clausen, Dustin},
  title         = {Duality and linearization for {$p$}-adic Lie groups},
  year          = {2025},
  eprint        = {2506.18174},
  archivePrefix = {arXiv},

  url           = {https://arxiv.org/abs/2506.18174},
}

@misc{DerivedStoneEmbedding,
  author        = {Kaminski, A.},
  title         = {Derived Stone embedding},
  year          = {2024},
  eprint        = {2410.24077},
  archivePrefix = {arXiv},
  url           = {https://arxiv.org/abs/2410.24077},
}

@misc{BH19,
  author       = {Barwick, C. and Haine, P.},
  title        = {Pyknotic objects, I: Basic notions},
  howpublished = {arXiv:1904.09966},
  year         = {2019},
  eprint       = {1904.09966},
  archivePrefix= {arXiv},
  primaryClass = {math},
  url          = {https://arxiv.org/abs/1904.09966}
}

@book{serre,
  author    = {Serre, Jean-Pierre},
  title     = {Lie Algebras and Lie Groups},
  subtitle  = {1964 Lectures given at Harvard University},
  series    = {Lecture Notes in Mathematics},
  volume    = {1500},
  edition   = {2},
  publisher = {Springer Berlin, Heidelberg},
  year      = {1992},
  doi       = {10.1007/978-3-540-70634-2},
  isbn      = {978-3-540-55008-2},
  note      = {Originally published by W. A. Benjamin, Inc., New York, 1965. eBook ISBN: 978-3-540-70634-2}
}

@misc{ConCom,
  title        = {Condensed Mathematics and Complex Geometry},
  author       = {Clausen, Dustin and Scholze, Peter},
  year         = {2022},
  month        = apr,
  note         = {Lecture notes (course jointly between Bonn and Copenhagen, Summer 2022)},
  institution  = {Max Planck Institute for Mathematics (MPIM), Bonn},
  url          = {https://people.mpim-bonn.mpg.de/scholze/Complex.pdf},
  urldate      = {2026-01-29}
}

@article{loc,
  title         = {General non-commutative locally compact locally Hausdorff Stone duality},
  author        = {Bice, Tristan and Starling, Charles},
  journal       = {Advances in Mathematics},
  volume        = {341},
  pages         = {40--91},
  year          = {2019},
  doi           = {10.1016/j.aim.2018.10.031},
  eprint        = {1803.00394},
  archivePrefix = {arXiv},
  primaryClass  = {math.OA},
  url           = {https://arxiv.org/abs/1803.00394},
  urldate       = {2026-02-26}
}

@misc{erratum,
  title        = {ERRATUM TO "p-ADIC HODGE THEORY FOR RIGID-ANALYTIC
VARIETIES"},
  author       = {Scholze, Peter},

  url          = {https://www.math.uni-bonn.de/people/scholze/pAdicHodgeErratum.pdf},
}

\end{document}